\documentclass[twoside,12pt]{article}
\usepackage{color}

\usepackage{graphics}
\usepackage{graphicx}

\usepackage{amssymb,amsfonts}
\usepackage{amsmath}
\usepackage{amsthm}

\graphicspath{{fig/},{figure/},{../fig/}}

\usepackage[all]{xy}
\usepackage{tikz}
\usetikzlibrary{arrows,chains,matrix,positioning,scopes}

\tikzset{>=stealth',every on chain/.append style={join},
         every join/.style={->}}
\tikzstyle{labeled}=[execute at begin node=$\scriptstyle,
   execute at end node=$]

\renewcommand{\arraystretch}{1.35}

\def\binomh#1#2{ \scalebox{.3}[1.2]{\textbf{)}}{\genfrac{}{}{0pt}{}{#1}{#2}}\scalebox{.3}[1.2]{\textbf{(}} }

\def\binomhh#1#2{ \scalebox{.4}[1.7]{\textbf{)}}{\genfrac{}{}{0pt}{}{#1}{#2}}\scalebox{.4}[1.7]{\textbf{(}} }

\def\ve{\varepsilon}
\def\vp{\varphi}

\newtheorem{theorem}{Theorem}
\newtheorem{lemma}{Lemma}[section]
\newtheorem{cor}{Corollary}
\newtheorem{rem}{Remark}

\title{\bf Hyperbolic Pascal triangles 
}
\author{Hacene Belbachir\footnote{STHB, Faculty of Mathematics, Algeria.  \textit{hbelbachir@usthb.dz}}, 
 L\'aszl\'o N\'emeth\footnote{University of West Hungary,  Institute of Mathematics, Hungary. \textit{nemeth.laszlo@emk.nyme.hu}}, 
 L\'aszl\'o Szalay\footnote{University of West Hungary,  Institute of Mathematics, Hungary.  \textit{szalay.laszlo@emk.nyme.hu}}}
\date{\today}

\begin{document}

\maketitle

\begin{abstract}
In this paper, we introduce a new generalization of Pascal's triangle. The new object is called the hyperbolic Pascal triangle since the mathematical background goes back to regular mosaics on the hyperbolic plane. We describe precisely the procedure of how to obtain a given type of hyperbolic Pascal triangle from a mosaic. Then we study certain quantitative properties such as the number, the sum, and the alternating sum of the elements of a row. Moreover, the pattern of the rows, and the appearence of some binary recurrences in a fixed hyperbolic triangle are investigated. \\[1mm]
\end{abstract}

\section{Introduction}\label{sec:introduction} 
There exist several variations of Pascal's arithmetic triangle (see, for instance, \cite{B} or \cite{BSz}).
This study provides a new generalization. The innovation is to extend the link between the infinite graph corresponding to the regular square Euclidean mosaic and classical Pascal's triangle to the hyperbolic plane, which contains infinitely many types of regular mosaics. The purpose of this paper is to describe this generalization and examine certain properties of the hyperbolic Pascal triangles. From one point we will focus on only one regular square mosaic given by the pair $\{4,q\}$, $q\ge5$. Later we will specialize in the case $q=5$. We remark, that in this paper the terminology Pascal's triangle is used only for Blaise Pascal's original triangle, while the expression Pascal triangle (without apostrophe) means a variation or generalization of Pascal's triangle. 

Let $p$ and $q$ denote two positive integers satisfying $p\ge3$, $q\ge3$, and consider the regular mosaic characterized by the Schl\"afli's symbol $\{p,q\}$. The parameters indicate that for regular $p$-gons of cardinality exactly $q$ meet at each vertex. 
If $(p-2)(q-2)=4$,  the mosaic belongs to the Euclidean plane, if $(p-2)(q-2)>4$ or $(p-2)(q-2)<4$, the mosaic is located on the hyperbolic plane or on the sphere, respectively (see, for instance \cite{C}). From now on we exclude the case $(p-2)(q-2)<4$, since we intend to deal only with infinite mosaics. 

The vertices and the edges of a given mosaic determine an infinite graph ${\cal G}$. Fix an arbitrary vertex, say $V_0$. Obviously, the $q$-fold rotation symmetry around $V_0$ and $q$ different mirror symmetries of the mosaic induce a $q$-fold rotation symmetry and mirror symmetries on the graph. Consequently, in accordance with the object of examination, it is sufficient to consider only a certain subgraph ${\cal G}^\prime$ of ${\cal G}$. In order to introduce hyperbolic Pascal triangles, now we define our own ${\cal G}^\prime$ precisely.

Restrict the mosaic to an unbounded part ${\cal P}$ of itself by defining first the border of ${\cal P}$, which contains regular $p$-gons as follows. The border splits the mosaic into two more parts, and the part which contains just one symmetric line among the $q$ different symmetric lines through the vertex $V_0$, together with the border gives ${\cal P}$. 
%{\color{red} (${\cal P}$ is convex if $q>3$. Nem is kell bele.)}

Taking a regular $p$-gon which fits to $V_0$, we consider it as the $0^{th}$ (base) cell of the border. Beginning from $V_0$, figure the edges of the base $p$-gon by $1$, $2$, \ldots, $p$  anti-clockwise (see Figure~\ref{fig:border}).  For the case of even and odd $p$, let first $p=2k$ $(k\geq 2)$. Mirror the base cell across the edge $k$ to obtain the first cell of the left hand side part of the border. We assume that the reflection transmits the figuring of the edges as well, for instance the edge $k$ is common of the $0^{th}$ and $1^{st}$ cells.  Then to gain the second cell we reflect the first $p$-gon through the edge $2k$ of the first cell. And so on, we always mirror the $\ell^{th}$ cell through the edge $k$ or $2k$ to obtain the $(\ell+1)^{th}$ element of the left hand side of the border. Getting back to $V_0$ and cell $0$, one can get the right hand side part of the border in a similar way by reflecting the appropriate elements across the edges $k+1$ and $1$, alternately. 
If $p=2k+1\ge5$, first we reflect the base cell across the edge $k+1$, and after then we can get the border in an analogous way to the case of even $p$ if we apply reflections through the edges $1$, $k+2$ (LHS) and $2k+1$, $k$ (RHS), alternately in both cases. 
If $p=3$, then after reflecting the base cell of the mosaic across the edge 2, we construct the LHS and RHS borders by consecutive reflections across edges $1,3,2,1,3,2,\dots$, and edges $3,1,2,3,1,2,\dots$, respectively. 
Figure \ref{fig:border} shows the construction of the border of $\cal P$ (and $\cal P$ itself when one joins the blue and yellow parts), first if $p$ is even, second if $p\ge5$ is odd, and then when $p=3$.  Moreover, Figure~\ref{fig:mosaic_45_pascal} illustrates the case $\{4,5\}$.

\begin{figure}[h!]
 \centering
 \includegraphics{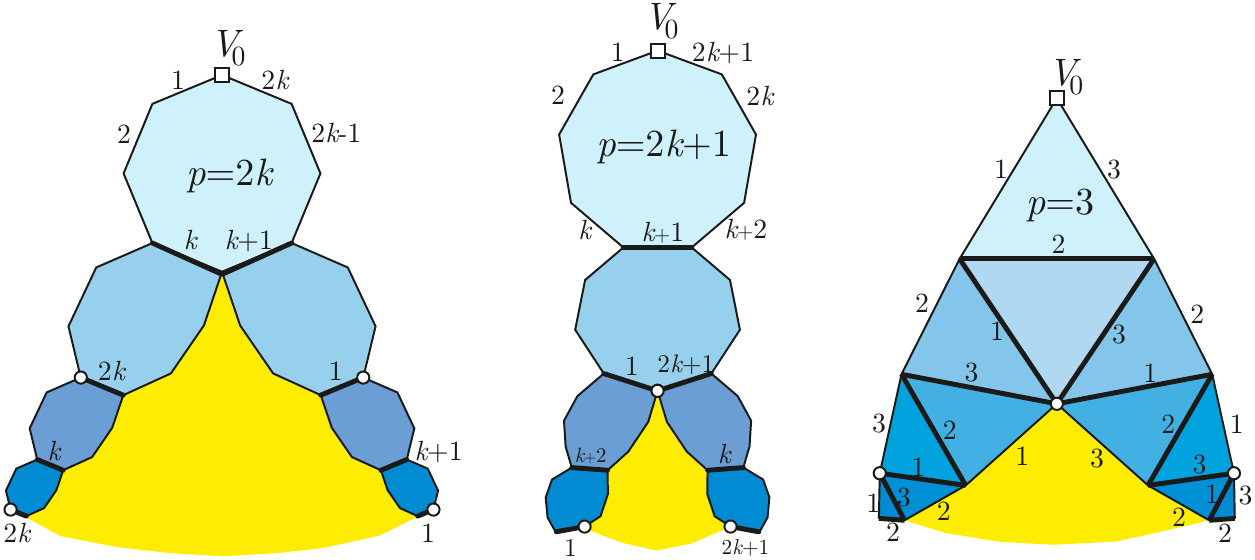}
 \caption{Construction of the border of $\cal{P}$}
 \label{fig:border}
\end{figure} 

Since ${\cal P}$ is a definite part of the mosaic, it eliminates a subgraph ${\cal G}_{\cal P}$ in the graph ${\cal G}$ of the whole mosaic. In the sequel, we always consider only this graph  ${\cal G}_{\cal P}$.  Assume that all edges have unit length. The distance of an arbitrary vertex $V$ and $V_0$ is the length of the shortest path between them.
%In order to investigate the quantitative properties of the graph it is sufficient to consider only a given subgraph of ${\cal G}$.
It is clear, that the shortest path from $V_0$ to any vertex of the outer boundary of ${\cal P}$ is unique, and passing through on the outer boundary itself. Inside of ${\cal P}$ there are always at least two shortest path from $V_0$ to any $V$.       
Let us label all vertices $V$ of ${\cal G}_{\cal P}$ by the number of distinct shortest paths from $V_0$ to $V$. As the first illustration
take the Euclidean squared mosaic (i.e.~$\{p,q\}=\{4,4\}$). Then the subgraph ${\cal G}_{\cal P}$ with its labelling returns with Pascal's original triangle (see Figure \ref{fig:mosaic_44_pascal}). 

\begin{figure}[h!]
 \centering
 \includegraphics[width=5.5cm]{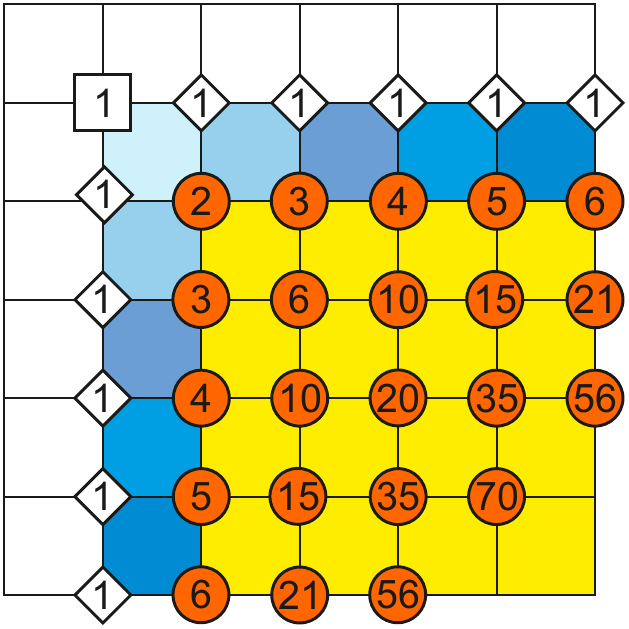}
 \caption{Pascal's triangle on the Euclidean mosaic $\{4,4\}$}
 \label{fig:mosaic_44_pascal}
\end{figure}

The other two Euclidean regular mosaics have no great interest since they are also associated with Pascal's triangle (see Figure \ref{fig:mosaic_3663_pascal}). 
In case of $\{3,6\}$ the appearance is direct, while $\{6,3\}$ displays all rows twice.
But any hyperbolic mosaic $\{p,q\}$, in the manner we have just described, leads to a so-called hyperbolic Pascal triangle (see the mosaic $\{4,5\}$, Figure \ref{fig:mosaic_45_pascal}). Thus the Euclidean mosaics and the hyperbolic mosaics together, based on the idea above, provide a new generalization of Pascal's triangle. In the next section, we narrow the spectrum of the investigations by concentrating only on the mosaics having Schl\"afli's symbol $\{4,q\}$. We do this to demonstrate the most spectacular class of hyperbolic Pascal triangles, when the connection to Pascal's triangle is the most obvious.

\begin{figure}[h!]
 \centering
 \includegraphics[width=13cm]{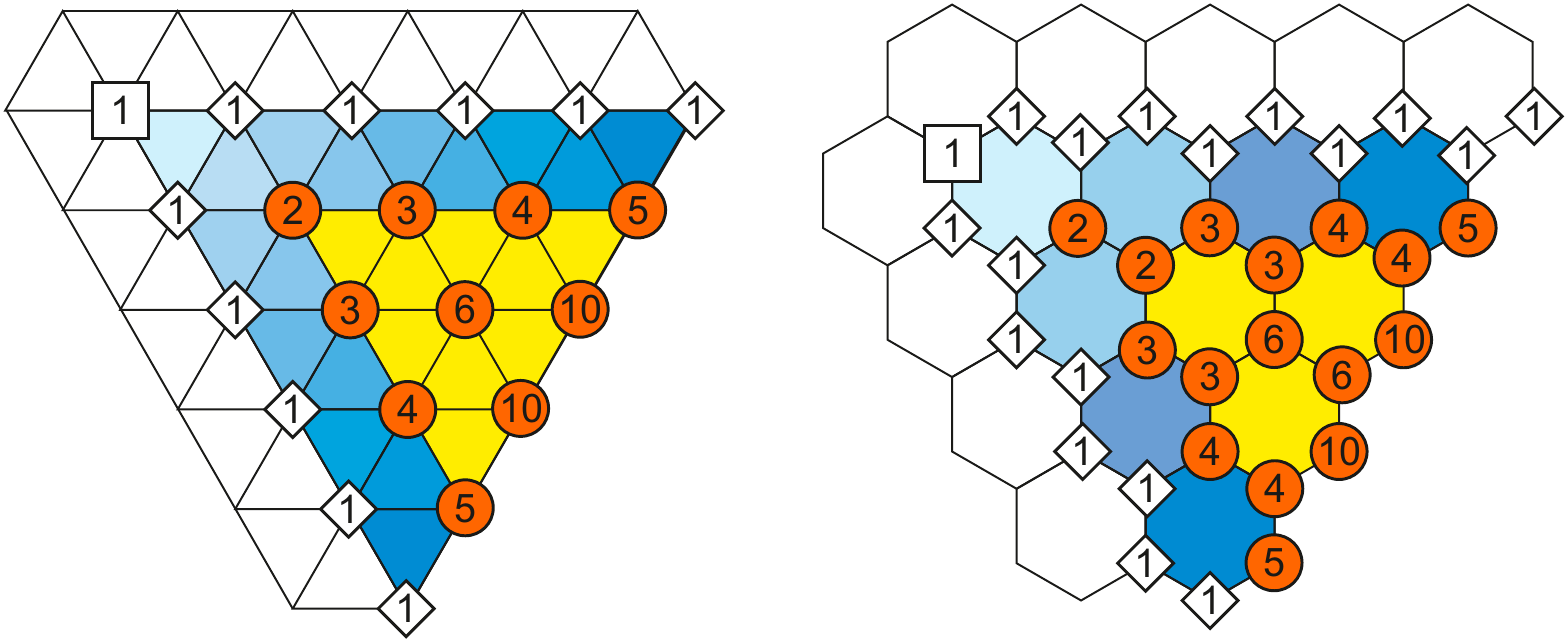}
 \caption{Pascal's triangle on the Euclidean mosaics $\{3,6\}$ and $\{6,3\}$ }
 \label{fig:mosaic_3663_pascal}
\end{figure}

\begin{figure}[h!]
 \centering
 \includegraphics{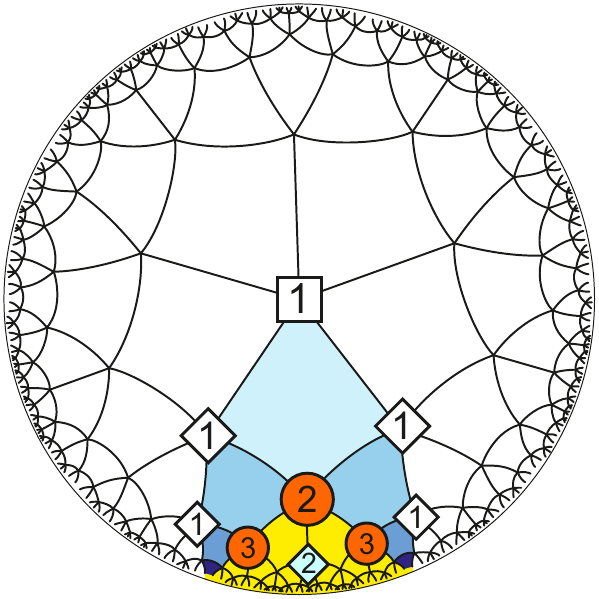}
 \caption{Pascal triangle on the hyperbolic mosaic $\{4,5\}$}
 \label{fig:mosaic_45_pascal}
\end{figure}

%--------------------------------------------------------------------------------------------------------

\section{Hyperbolic Pascal triangle linked to $\{4,q\}$}\label{sec:hpt4q} 

Let us fix a mosaic given by $\{4,q\}$, $q\ge5$. In order to have a more favourable object to study, now we introduce a graph ${\cal I}_{4,q}$, isomorphic to ${\cal G}_{\cal P}$ and leading immediately to the corresponding hyperbolic Pascal triangle. We ``reorganize'' the graph ${\cal G}_{\cal P}$ by defining rows (or levels) to get ${\cal I}_{4,q}$ as follows. The $0^{th}$ row of ${\cal I}_{4,q}$ contains only $V_0$ itself. Generally, the $j^{th}$ row ($j\ge1$) consists of the vertices having distance $j$ from $V_0$. According to the definition of the border of ${\cal P}$, the leftmost vertex in the $j^{th}$ row ($j\ge1$) has one edge upwards (to level $j-1$) connecting to the leftmost vertex of the $(j-1)^{th}$ row, further two edges downwards (to level $j+1$), one goes into the leftmost vertex of the $(j+1)^{th}$ row, the other is the edge played the role of mirror axis in constructing the border of ${\cal P}$. Thus, apart from $V_0$ which has degree 2, all the elements on the ``left leg'' have degree 3. (Note that in the whole mosaic, any vertex has degree $q$, hence $q-3$ neighbours of any element of the left leg are outside of ${\cal I}_{4,q}$.) The analogous observation is valid for the ``right leg'' of  ${\cal I}_{4,q}$ through mirror symmetry. The inner part of ${\cal I}_{4,q}$ consists of two kinds of vertices, although all of them have degree $q$. Generally, if $V$ is located in the $j^{th}$ row ($j\ge2$), then it has either two edges or one edge upwards (to the row $(j-1)$), and either $q-2$ or $q-1$ edges downwards (to the row $(j+1)$), respectively. 

The description above provides the following unforced algorithm to produce the graph ${\cal I}_{4,q}$. Locate $V_0$ into row 0. The first row contains only two vertices, say $V_{1,0}$ and $V_{1,1}$ from left to right, both of them are connected to $V_0$. Both $V_{1,0}$ and $V_{1,1}$ have two edges going down to the second row, but the second edge of $V_{1,0}$ and the first edge of $V_{1,1}$ meet in one vertex of the graph (in $V_{2,1}$). Hence row 2 admits three vertices: $V_{2,0}$, $V_{2,1}$ and $V_{2,2}$. We construct the third row, and then we are able to give the general rule.  Both $V_{2,0}$ and $V_{2,2}$ have two descendants, while $V_{2,1}$ has $q-2$ descendants since it owns two edges upwards. The second edge coming from $V_{2,0}$ and the first edge from $V_{2,1}$ meet in one vertex of the third row. The same happens with the $(q-2)^{th}$ edge of $V_{2,1}$ and first edge of $V_{2,2}$. The remaining $q-4$ edges of $V_{2,1}$ are linked to $q-4$ distinct vertices of row 3. Thus, in this row, apart from the leftmost and rightmost elements we can observe two types of vertices. Some of the vertices (called ``Type A'' for convenience) have two ascendants and $q-2$ descendants, the others (``Type B'') have one ascendant and $q-1$ descendants. This is true for all rows below, and shows the general method of drawing. Going along the vertices of the $j^{th}$ row, according to type of the elements (winger, $A$, $B$), we draw appropriate number of edges downwards (2, $q-2$, $q-1$, respectively). Neighbour edges of two neighbour vertices of the $j^{th}$ row meet in the $(j+1)^{th}$ row, constructing a vertex with type $A$. The other descendants of row $j$ in row $j+1$ have type $B$ (see Figure  \ref{fig:Pascal_growing4q}).

\begin{figure}[h!]
 \centering
 \includegraphics[width=7cm]{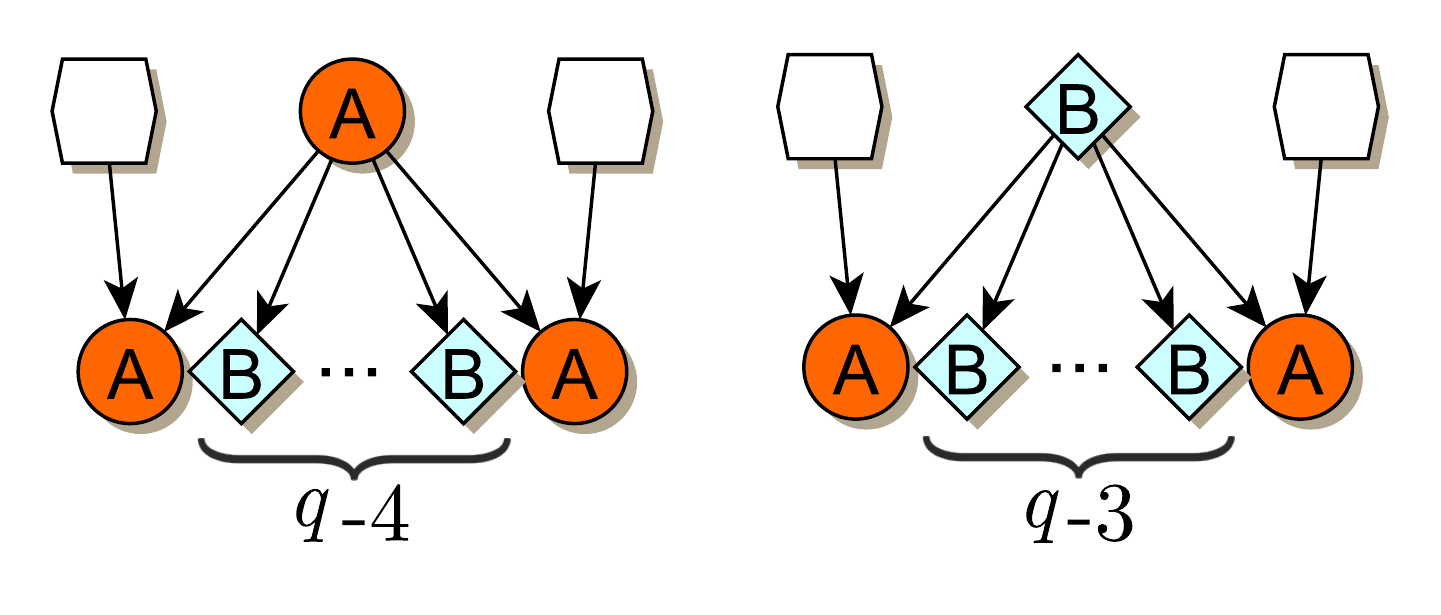}
 \caption{Growing method in the case $\{4,q\}$}
 \label{fig:Pascal_growing4q}
\end{figure} 

Finally, if we indicate the labels of the vertices (by marking the number of the distinct shortest paths to reach them from $V_0$), we obtain a new object: a hyperbolic Pascal triangle with parameters $\{4,q\}$. Figure~\ref{fig:Pascal_46_layer5} illustrates the labelled graph ${\cal I}_{4,6}$ up to the fifth row. The Pascal triangle appears on graph ${\cal I}_{4,q}$ will also be denoted by ${\cal I}_{4,q}$. In the sequel, similarly to the notation of classical binomial coefficients, we denote by $\binomh{j}{k}$ the
integer which is the $k^{th}$ element in row $j$ of the given hyperbolic Pascal triangle.  

\begin{figure}[h!]
 \centering
 \includegraphics[width=0.99\linewidth]{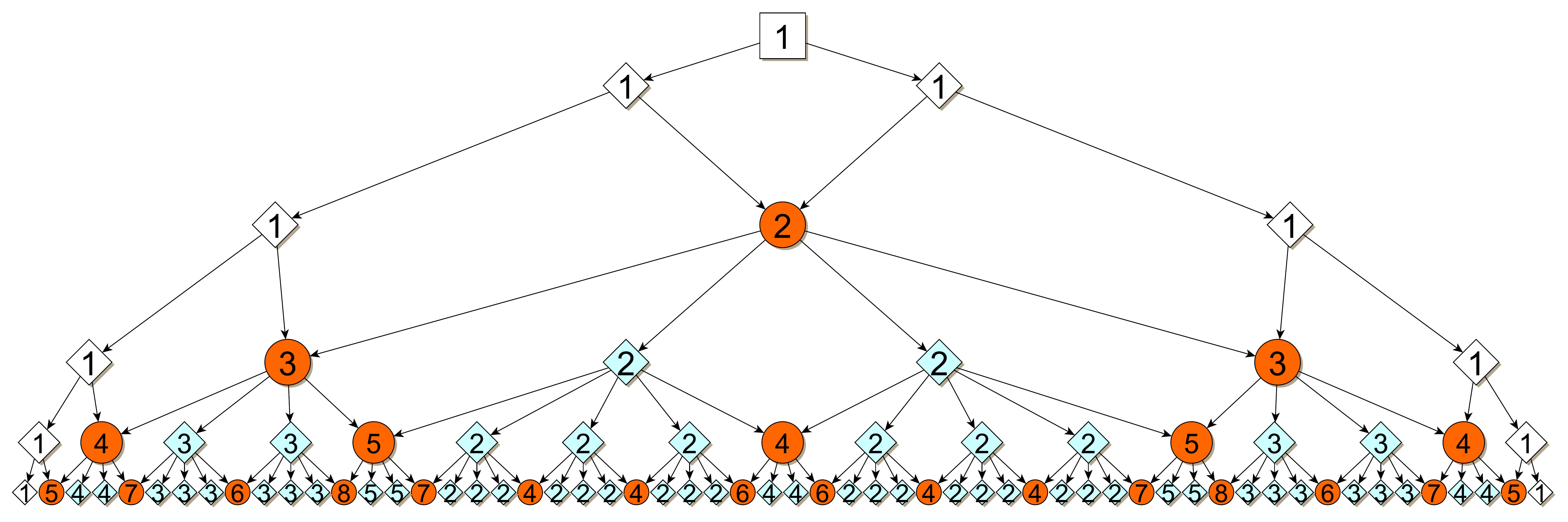}
 \caption{First rows of the hyperbolic Pascal triangle linked to $\{4,6\}$}
 \label{fig:Pascal_46_layer5}
\end{figure}

Clearly, the calculation of the labels is the following: it either keeps the label of its unique ascendant or is the sum of the labels of its two ascendants. This property shows directly that we arrived at a generalization of the original Pascal's triangle. If we consider again, for a moment the Euclidean mosaic $\{4,4\}$, in this case no elements with type $B$ exist (see Figure~\ref{fig:mosaic_44_pascal} again). 

In studying the quantitative properties of the hyperbolic Pascal triangle ${\cal I}_{4,q}$, first we determine the number of the elements of the $n^{th}$ row of the graph. Denote by $a_n$ and $b_n$ the number of vertices of type $A$ and $B$, respectively, further let 
\begin{equation}\label{sn}
s_n=a_n+b_n+2, 
\end{equation}
which gives the total number of the vertices of row $n\ge1$. Recall, that $q\ge5$.

\begin{theorem}\label{theorem:numvertex4q}  
The three sequences $\{a_n\}$, $\{b_n\}$ and $\{s_n\}$ can be described by the same ternary homogenous recurrence relation 
\begin{equation}\label{recur1}
x_n=(q-1)x_{n-1}-(q-1)x_{n-2}+x_{n-3}\qquad (n\ge4),
\end{equation}
the initial values are 
$$
a_1=0,\;a_2=1,\;a_3=2,\qquad b_1=0,\;b_2=0,\;b_3=q-4,\qquad s_1=2,\;s_2=3,\;s_3=q.
$$ 
Moreover, the explicit formulae
\begin{eqnarray}\label{abcn4q}
a_n&=&\left(\frac{2-q}{2}+\frac{q^2-4q+2}{2q(q-4)}\sqrt{D}\right)\alpha_q^n+\left(\frac{2-q}{2}-\frac{q^2-4q+2}{2q(q-4)}\sqrt{D}\right)\beta_q^n+1, \\
b_n&=&\left(\frac{q-3}{2}+\frac{1-q}{2q}\sqrt{D}\right)\alpha_q^n+\left(\frac{q-3}{2}-\frac{1-q}{2q}\sqrt{D}\right)\beta_q^n-1, \nonumber \\
s_n&=&\left(-\frac{1}{2}+\frac{q-2}{2q(q-4)}\sqrt{D}\right)\alpha_q^n+\left(-\frac{1}{2}-\frac{q-2}{2q(q-4)}\sqrt{D}\right)\beta_q^n+2 \nonumber
\end{eqnarray}
are valid for $n\ge1$, where 
$$
D=q^2-4q,\qquad\alpha_q=\frac{(q-2)+\sqrt{D}}{2},\qquad \beta_q=\frac{(q-2)-\sqrt{D}}{2}.
$$
\end{theorem}

\begin{proof}
Any two neighbour elements of row $n$ establish exactly one element of type $A$ in row $(n+1)$, i.~e.
\begin{equation*}%\label{a4q}
a_{n+1}=a_n+b_n+1.
\end{equation*}
Obviously, the number of the vertices of type $B$ in row $(n+1)$ is given by
\begin{equation*}%\label{b4q}
b_{n+1}=(q-4)a_n+(q-3)b_n.
\end{equation*}
For the sequences  $\{a_n\}$ and $\{b_n\}$, Lemma \ref{lemma:2seq} returns with (\ref{recur1}).

In case of the sequence $\{s_n\}$, first apply the definition (\ref{sn}) for $s_n$ and $s_{n+1}$. Then
$$
(q-1)(s_{n+1}-s_n)=(q-1)(a_{n+1}-a_n)+(q-1)(b_{n+1}-b_n)
$$ 
follows. Adding $s_{n-1}=a_{n-1}+b_{n-1}+2$ to both sides, we obtain $s_{n+2}$ on the right hand side. Hence the first part of the proof is complete.

Consider now the common characteristic polynomial 
$$p(x)=x^3-(q-1)x^2+(q-1)x-1=(x-1)(x^2-(q-2)x+1)$$ 
of the three homogenous ternary recursive sequences. It has real zeros 
$$
\alpha_q=\frac{(q-2)+\sqrt{q^2-4q}}{2},\qquad \beta_q=\frac{(q-2)-\sqrt{q^2-4q}}{2},\qquad \gamma_q=1.
$$
Consequently, for instance, $a_n$ is a linear combination of the terms $\alpha_q^n$, $\beta_q^n$ and $1$. In order to determine the coefficients in the linear combination we use the initial values $a_1$, $a_2$ and $a_3$. The three equations provide (\ref{abcn4q}). The treatment for $b_n$ is analogous. Finally, one can easily check the formula for $s_n$ by evaluating $s_n=a_n+b_n+2$.
\end{proof}
\medskip

It is well known that in Pascal's triangle the sum of all the binomial coefficients of the $n^{th}$ row is $\xi_n=2^n$. Therefore, the question arises naturally about the sum of the elements of row $n$ in the hyperbolic Pascal triangle linked to $\{4,q\}$. Although we do not know explicit formula for the elements themselves, we are able to determine the sum. The result can also be described by homogenous ternary recurrences. Comparing it to $\xi_n$, which is a geometric progression, or in other words a homogenous linear recurrence of order 1, the two sums are related.

Let $\hat{a}_n$, $\hat{b}_n$ and $\hat{s}_n$ denote the sum of type $A$, type $B$ and all elements of the $n^{th}$ row, respectively. We will justify the following statements. 

\begin{theorem}\label{theorem:sumvertex4q}  
The three sequences $\{\hat{a}_n\}$, $\{\hat{b}_n\}$ and $\{\hat{s}_n\}$ can be described by the same ternary homogenous recurrence relation 
\begin{equation*}
x_n=qx_{n-1}-(q+1)x_{n-2}+2x_{n-3}\qquad (n\ge4),
\end{equation*}
the initial values are 
$$
\hat{a}_1=0,\;\hat{a}_2=2,\;\hat{a}_3=6,\qquad \hat{b}_1=0,\;\hat{b}_2=0,\;\hat{b}_3=2(q-4),\qquad \hat{s}_1=2,\;\hat{s}_2=4,\;\hat{s}_3=2q.
$$ 
Moreover, 
\begin{eqnarray*}%\label{abcn4q2}
\hat{a}_n&=&\left(\frac{1-q}{2}+\frac{q^2-2q-3}{2(q^2-2q-7)}\sqrt{D}\right)\hat{\alpha}_q^n+\left(\frac{1-q}{2}-\frac{q^2-2q-3}{2(q^2-2q-7)}\sqrt{D}\right)\hat{\beta}_q^n+2, \\
\hat{b}_n&=&\left(\frac{q-2}{2}-\frac{q^2-3q-2}{2(q^2-2q-7)}\sqrt{D}\right)\hat{\alpha}_q^n+\left(\frac{q-2}{2}+\frac{q^2-3q-2}{2(q^2-2q-7)}\sqrt{D}\right)\hat{\beta}_q^n-2, \\
\hat{s}_n&=&\left(-\frac{1}{2}+\frac{q-1}{2(q^2-2q-7)}\sqrt{D}\right)\hat{\alpha}_q^n+\left(-\frac{1}{2}-\frac{q-1}{2(q^2-2q-7)}\sqrt{D}\right)\hat{\beta}_q^n+2 
\end{eqnarray*}
hold, where 
$$
D=q^2-2q-7,\qquad\hat{\alpha}_q=\frac{(q-1)+\sqrt{D}}{2},\qquad \hat{\beta}_q=\frac{(q-1)-\sqrt{D}}{2}.
$$
\end{theorem}

\begin{proof}
By the construction rule of the triangle, together with the definition of the sequences above, we see that
\begin{equation*}%\label{hatab4q}
\hat{a}_{n+1}=2\hat{a}_n+2\hat{b}_n+2,\qquad \hat{b}_{n+1}=(q-4)\hat{a}_n+(q-3)\hat{b}_n,
\end{equation*}
further
\begin{equation*}%\label{hatsb4q}
\hat{s}_{n}=\hat{a}_n+\hat{b}_n+2.
\end{equation*}
Now we apply the machinery of the proof of Theorem \ref{theorem:numvertex4q} to obtain the homogenous ternary recurrence relation of this theorem. The common characteristic polynomial
$$
\hat{p}(x)=x^3-qx^2+(q+1)x-2=(x-1)(x^2-(q-1)x+2)
$$
of the three sequences has zeros $\hat{\alpha}_q$, $\hat{\beta}_q$ and $1$. A suitable linear combination of their $n^{th}$ power provide the explicit formulae of Theorem \ref{theorem:sumvertex4q}.
\end{proof}

%--------------------------------------------------------------------------------------------------------

\section{The hyperbolic Pascal triangle linked to $\{4,5\}$}\label{sec:geom} 

Figure \ref{fig:Pascal_layer5}, in part, shows the object of the forthcoming scrutiny. The elements of type $A$ are displayed in circle, type $B$ in square.
 
\begin{figure}[h!]
 \centering
 \includegraphics[width=0.83\linewidth]{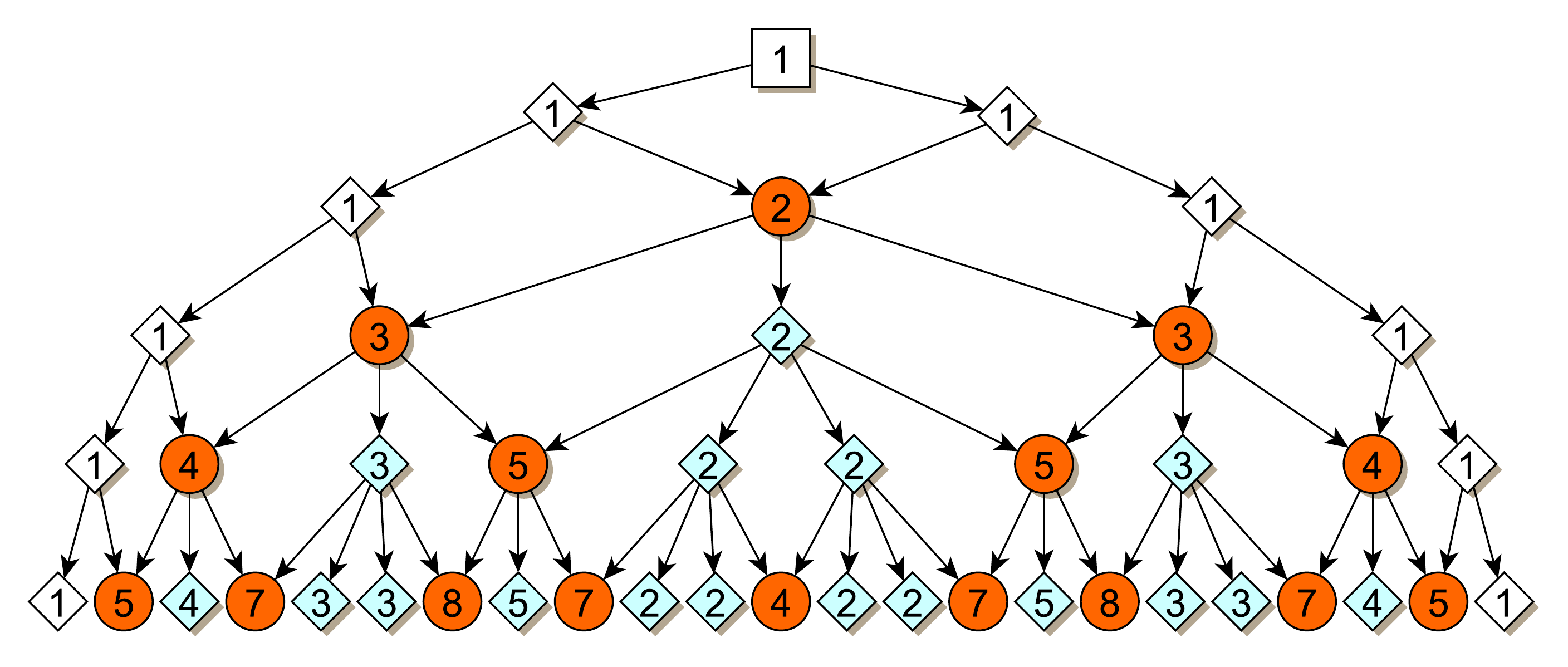}
 \caption{First rows of the hyperbolic Pascal triangle $\{4,5\}$}
 \label{fig:Pascal_layer5}
\end{figure}

First survey some basic properties of graph ${\cal I}_{4,5}$ and the corresponding Pascal triangle.
Obviously, there triangle possesses a vertical symmetry. In a given row, apart from the wingers, between two neighbour elements of type $A$ only $B$ or $BB$ are possible. Between two $B$'s either is nothing or one $A$ appears.

Now we state two consequences of Theorem \ref{theorem:numvertex4q}.

\begin{cor}
The recursive rule
\begin{equation}\label{r45}
x_n=4x_{n-1}-4x_{n-2}+x_{n-3}\qquad (n\ge4),
\end{equation}
holds for the three sequences $\{a_n\}$, $\{b_n\}$ and $\{s_n\}$,
the initial values are 
$$
a_1=0,\;a_2=1,\;a_3=2,\qquad b_1=0,\;b_2=0,\;b_3=1,\qquad s_1=2,\;s_2=3,\;s_3=5.
$$ 
Moreover, the explicit formulae
\begin{eqnarray*}%\label{abcn4q3}
a_n&=&\left(-\frac{3}{2}+\frac{7}{10}\sqrt{5}\right)\alpha_5^n+\left(-\frac{3}{2}-\frac{7}{10}\sqrt{5}\right)\beta_5^n+1, \\
b_n&=&\left(1-\frac{2}{5}\sqrt{5}\right)\alpha_5^n+\left(1+\frac{2}{5}\sqrt{5}\right)\beta_5^n-1, \\
s_n&=&\left(-\frac{1}{2}+\frac{3}{10}\sqrt{5}\right)\alpha_5^n+\left(-\frac{1}{2}-\frac{3}{10}\sqrt{5}\right)\beta_5^n+1 
\end{eqnarray*}
are valid, where $\alpha_5=(3+\sqrt{5})/2$, $\beta_5=(3-\sqrt{5})/2$.
\end{cor}

We note, that the ratio of the numbers of elements of two consecutive rows tends to $\alpha_5\approx2.618$, which means exponential growing (in Pascal's triangle this ratio tends to 1). 
The recurrence (\ref{r45}) and the initial values of the sequence $\{s_n\}$ give immediately the parity of the numbers $s_n$ as follows.

\begin{cor}
$$
s_n\equiv\left\{
\begin{array}{lll}
0 \;\;(\bmod \;2)& {\rm if} & n=3t+1, \\
1 \;\;(\bmod \;2)& {\rm if} & n\ne3t+1.
\end{array}
\right.
$$
\end{cor}

\subsection{Alternating sum in the hyperbolic Pascal triangle ${\cal I}_{4,5}$}\label{subsec:geom} 

In Theorem \ref{theorem:sumvertex4q} the sum of the elements of a row has been obtained. In case of triangle ${\cal I}_{4,5}$ we are able to determine the alternating sums. At the end of this subsection we will replace the weights 1 and $-1$ by arbitrary weights $v$ and $w$. We remark that such weighted sums under more general conditions in one generalization of Pascal's triangle have been examined by \cite{A}, where the direction of the summary was arbitrary.

\begin{theorem}
$$
\widetilde{s}_n=\sum_{i=0}^{s_n-1}(-1)^i\,\binomhh{n}{i}=
\left\{
\begin{array}{lll}
0 & {\rm if} & n=3t+1, \; n\ge1,\\
2 & {\rm if} & n\ne3t+1, \;n\ge5,
\end{array}
\right.
$$
further $\widetilde{s}_0=1$, $\widetilde{s}_2=0$, $\widetilde{s}_3=-2$. 
\end{theorem}

\begin{table}[!ht]
  \centering \setlength{\tabcolsep}{0.4em}
\begin{tabular}{|c|r|r|r|r|r|r|r|r|r|r|r|r|r|}
  \hline
 $n$   &  0  &  1 & 2 & 3 & 4 & 5 & 6  & 7  & 8   & 9   & 10  & 11   & 12    \\ \hline \hline
 $\widetilde{s}_n^{(A)}$ &  0  &  0 & -2 & -6 & 0 & 2 & 2 & 0 & 2 & 2 & 0 & 2  & 2  \\ \hline
 $\widetilde{s}_n^{(B)}$ &  0  &  0 & 0 & 2 & 0 & -2& -2 & 0 & -2 & -2 & 0& -2  & -2  \\ \hline
 $\widetilde{s}_n$ &  1  &  0 & 0 & -2 & 0 & 2 & 2  & 0  & 2   & 2   & 0   & 2    & 2      \\ \hline
 \end{tabular}
\caption{\emph{Alternating sums with $n\le12$}\label{table:altsum_of_values}}
\end{table}

\begin{proof}
By the vertical symmetry, the statement (i.e.~$\widetilde{s}_n=0$) is trivial if $n=3t+1$ since $s_{3t+1}$ is even. 

Assume now that $n$ is sufficiently large and $n\ne3t+1$. Thus the sum $\widetilde{s}_n$ consists of odd number of summands. Let $\widetilde{s}_n^{(A)}$ be the subsum of the alternate sum $\widetilde{s}_n$ what is restricted only to the elements of type $A$. Introducing $\widetilde{s}_n^{(B)}$ analogously, together with  $\binomh{n}{0}=\binomh{n}{s_n-1}=1$, clearly
$$
\widetilde{s}_n=\widetilde{s}_n^{(A)}+\widetilde{s}_n^{(B)}+2
$$
holds.

The backbone of the proof is to observe the influence of the elements of row $n$ to $\widetilde{s}_{n+3}$. We separate the contribution of each $\binomh{n}{i}$ to $\widetilde{s}_{n+3}$ individually, and then consider their superposition. In order to achieve the computations we will analyse three figures. Before, some useful notations are introduced.

Suppose that $x_A$ and $x_B$ denote the value of an element of type $A$ and $B$, respectively. Let their contribution to $\widetilde{s}_{n+3}$ be denoted by ${\cal H}_3(x_A)$ and ${\cal H}_3(x_B)$, respectively. Later it will be clear what we understand under contribution. Similarly, ${\cal H}_3(1)$ shows the contribution of a winger element of row $n$ (having value $1$). We also need the notations
${\cal H}_3^{(A)}(x_A)$ and ${\cal H}_3^{(B)}(x_A)$ for the contribution of the $A$ type element $x_A$ from the row $n$ to the alternate sum of the row $n+3$ restricted to the elements of type $A$ and $B$ in, respectively. Apparently,
$$
{\cal H}_3(x_A)={\cal H}_3^{(A)}(x_A)+{\cal H}_3^{(B)}(x_A).
$$
After then the meaning of ${\cal H}_3^{(A)}(x_B)$, ${\cal H}_3^{(B)}(x_B)$, moreover ${\cal H}_3^{(A)}(1)$ and ${\cal H}_3^{(B)}(1)$ are also clear. We even have
\begin{equation}\label{eq:febr}
{\cal H}_3(x_B)={\cal H}_3^{(A)}(x_B)+{\cal H}_3^{(B)}(x_B)\qquad{\rm and}\qquad {\cal H}_3(1)={\cal H}_3^{(A)}(1)+{\cal H}_3^{(B)}(1)+1.
\end{equation}

Put $\varepsilon=\pm1$. From Figure \ref{fig:HA_45} (where $x_A$ is shortened to $x$ for the simplicity), using the symmetry, we conclude
\begin{eqnarray*}
{\cal H}_3^{(A)}(x_A) &=& \ve\cdot(x+2x-3x-3x)\cdot2+\ve\cdot2x=-4\ve x=-4\ve x_A,\\
{\cal H}_3^{(B)}(x_A) &=& \ve\cdot(-x-x+x+2x+x-x)\cdot2=2\ve x=2\ve x_A.\end{eqnarray*}

\begin{figure}[h!]
 \centering
 \includegraphics[width=0.82\linewidth]{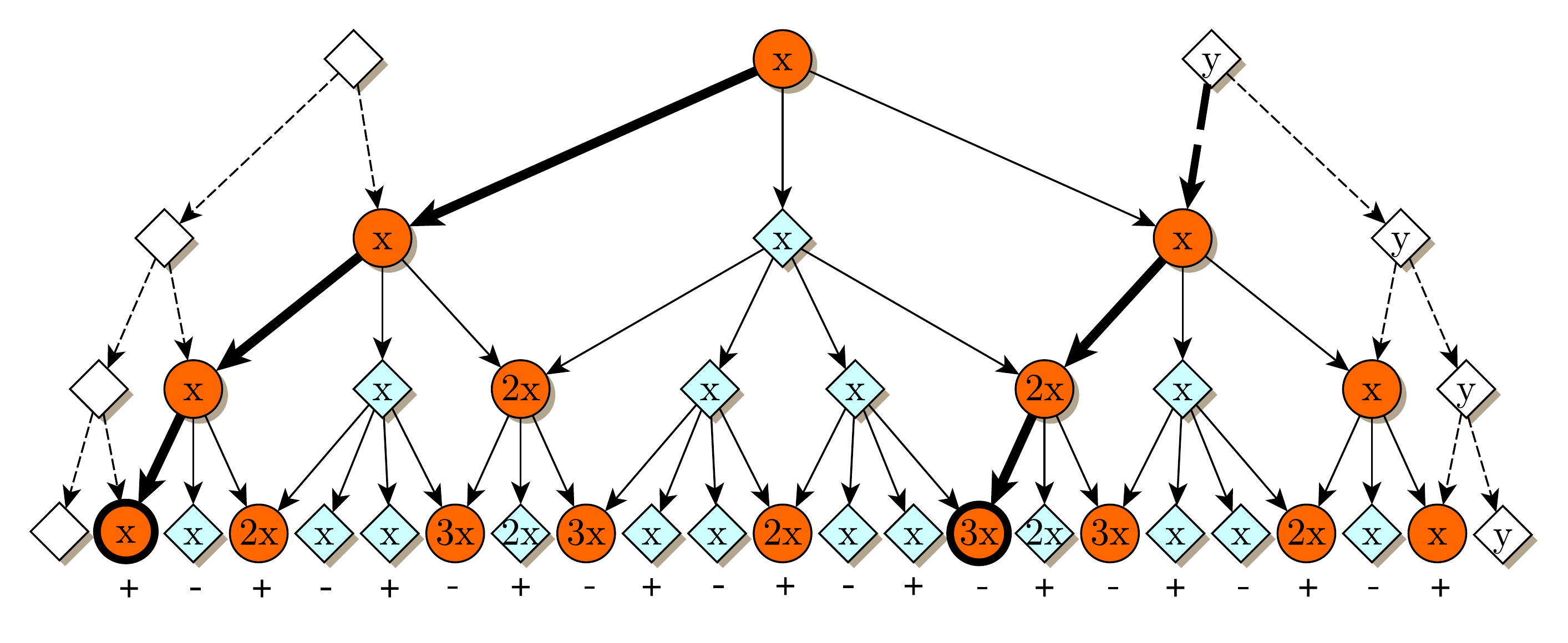}
 \caption{The influence ${\cal H}_3(x_A)$}
 \label{fig:HA_45}
\end{figure}

Thus ${\cal H}_3(x_A)=-2\ve x_A$. Assume that $x_A$, which is located in the row $n$, has the coefficient $\ve_1$ in the alternate sum of row $n$. Then the coefficient of the next element $y$ (see Figure \ref{fig:HA_45}) must be $-\ve_1$. It is crucial to observe, that there are further 12 elements ($x,2x,x,x,3x,2x,3x,x,x,2x,x,x$) between the leftmost elements of the influences of $x_A$ and $y$ in the $(n+3)^{th}$ row.  The change in the sign from $x_A$ to $y$ entails the same change in the sign from ${\cal H}_3(x_A)$ to ${\cal H}_3(y)$.  In other word, the signs of the alternate sum in row $n$ descent to the row $n+3$ in this manner.

Letting $\delta=\pm1$, a similar consequences come from Figure \ref{fig:HB_45} for an element of type $B$. More precisely,
\begin{eqnarray*}
{\cal H}_3^{(A)}(x_B) &=& \delta\cdot(x+2x-3x-3x+2x-3x)\cdot2=-8\delta x=-8\delta x_B,\\
{\cal H}_3^{(B)}(x_B) &=& \delta\cdot(-x-x+x+x+x-x-x+x)\cdot2+\delta\cdot2x=4\delta x=4\delta x_B.
\end{eqnarray*}

\begin{figure}[h!]
 \centering
 \includegraphics[width=0.99\linewidth]{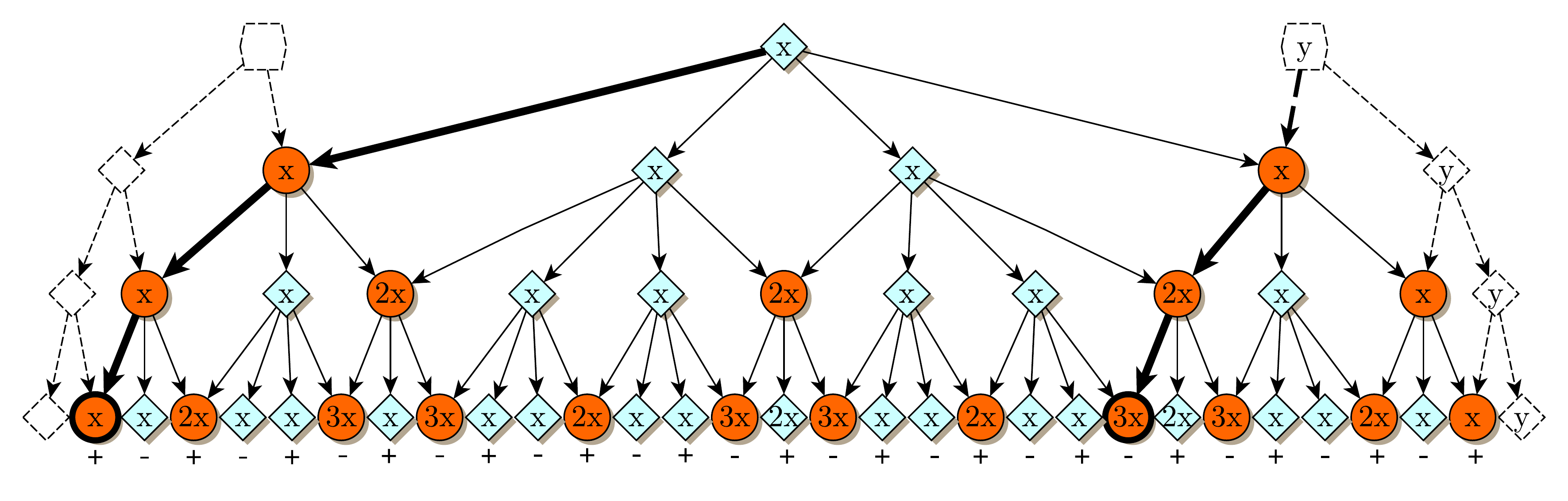}
 \caption{The influence ${\cal H}_3(x_B)$}
 \label{fig:HB_45}
\end{figure}

Hence ${\cal H}_3(x_B)=-4\delta x_B$. The inheritance of the sings of the alternate sum is also true here, by Figure \ref{fig:HB_45}, there are now 20 elements ($x,2x,x,\dots,2x,x,x$) between the leftmost elements of the influences of $x_A$ and $y$ in the $(n+3)^{th}$ row. 

Finally, Figure \ref{fig:H1_45} gives
\begin{eqnarray*}
{\cal H}_3^{(A)}(1) &=& -3-3+2+1=-3,\\
{\cal H}_3^{(B)}(1) &=& 2+1-1-1=1,
\end{eqnarray*}
and then (\ref{eq:febr}) implies ${\cal H}_3(1)=-1$. In this case, as $1$ is followed by $y$ in row $n$, the influence of $1$ is followed next by the influence of $y$.

\begin{figure}[h!]
 \centering
 \includegraphics[width=0.43\linewidth]{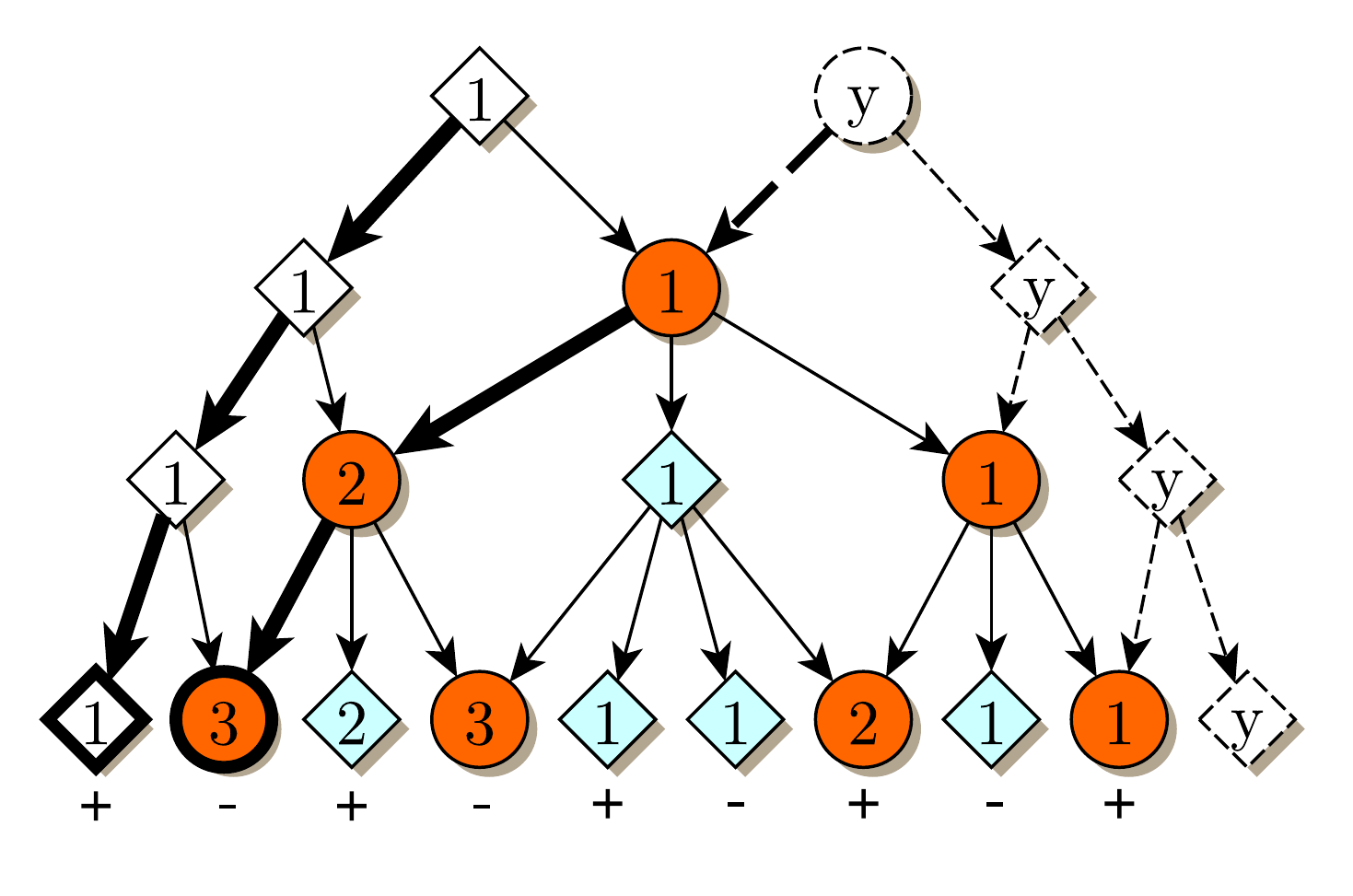}
 \caption{The influence ${\cal H}_3(1)$}
 \label{fig:H1_45}
\end{figure}

In the next step we derive the total influence of $\widetilde{s}_{n}^{(A)}$ (and then $\widetilde{s}_{n}^{(B)}$) to row $n+3$. Clearly, in $\widetilde{s}_{n}^{(A)}$ some elements have sign $-1$ (denoted by $x_A^{(-1)}$), the others have $1$ (denoted by $x_A^{(1)}$), so
$$
\widetilde{s}_{n}^{(A)}=\sum_{i=1}^{u_i}x_{A,i}^{(1)}\,+\,\sum_{j=1}^{u_j}(-1)x_{A,j}^{(-1)}.
$$
Thus
\begin{eqnarray*}
{\cal H}_3(\widetilde{s}_{n}^{(A)})&=&{\cal H}_3\left(\sum_{i=1}^{u_i}x_{A,i}^{(1)}\,+\,\sum_{j=1}^{u_j}(-1)x_{A,j}^{(-1)}\right)=\sum_{i=1}^{u_i}{\cal H}_3(x_{A,i}^{(1)})\,+\,\sum_{j=1}^{u_j}{\cal H}_3(-x_{A,j}^{(-1)}) \\
&=& \sum_{i=1}^{u_i}{\cal H}_3(x_{A,i}^{(1)})\,-\,\sum_{j=1}^{u_j}{\cal H}_3(x_{A,j}^{(-1)})\\
&=& \sum_{i=1}^{u_i}\left({\cal H}_3^{(A)}(x_{A,i}^{(1)})+{\cal H}_3^{(B)}(x_{A,i}^{(1)})\right)\,-\,\sum_{j=1}^{u_j}\left({\cal H}_3^{(A)}(x_{A,j}^{(-1)})+{\cal H}_3^{(B)}(x_{A,j}^{(-1)})\right)\\
&=& \left[\sum_{i=1}^{u_i}{\cal H}_3^{(A)}(x_{A,i}^{(1)})-\sum_{j=1}^{u_j}{\cal H}_3^{(A)}(x_{A,j}^{(-1)})\right]+\left[\sum_{i=1}^{u_i}{\cal H}_3^{(B)}(x_{A,i}^{(1)})-\sum_{j=1}^{u_j}{\cal H}_3^{(B)}(x_{A,j}^{(-1)})\right]\\
&=& \left[\sum_{i=1}^{u_i}(-4x_{A,i}^{(1)})-\sum_{j=1}^{u_j}(-4x_{A,j}^{(-1)})\right]+\left[\sum_{i=1}^{u_i}(2x_{A,i}^{(1)})-\sum_{j=1}^{u_j}(2x_{A,j}^{(-1)})\right]\\
&=& (-4)\left[\sum_{i=1}^{u_i}x_{A,i}^{(1)}-\sum_{j=1}^{u_j}x_{A,j}^{(-1)}\right]+2\left[\sum_{i=1}^{u_i}x_{A,i}^{(1)}-\sum_{j=1}^{u_j}x_{A,j}^{(-1)}\right]=\underbrace{-4\widetilde{s}_{n}^{(A)}}_{\rm{for}\;\rm{type}\; {\it A}}+\underbrace{2\widetilde{s}_{n}^{(A)}}_{\rm{for}\;\rm{type}\; {\it B}}\!\!\!\!.
\end{eqnarray*}
Similarly, we find
$$
{\cal H}_3(\widetilde{s}_{n}^{(B)})=\underbrace{-8\widetilde{s}_{n}^{(B)}}_{\rm{for}\;\rm{type}\; {\it A}}+\underbrace{4\widetilde{s}_{n}^{(B)}}_{\rm{for}\;\rm{type}\; {\it B}}\!\!\!\!.
$$
Hence
\begin{eqnarray*}
\widetilde{s}_{n+3}^{(A)}&=&-4\widetilde{s}_n^{(A)}-8\widetilde{s}_n^{(B)}+2\cdot(-3),\qquad (n\ge0),\\
\widetilde{s}_{n+3}^{(B)}&=&\;\;\;2\widetilde{s}_n^{(A)}+4\widetilde{s}_n^{(B)}+2\cdot1,\phantom{(-)}\qquad (n\ge0).
\end{eqnarray*}
Recalling Lemma \ref{lemma:2seq}, it implies $\widetilde{s}_{n+9}^{(A)}=\widetilde{s}_{n+6}^{(A)}$ and $\widetilde{s}_{n+9}^{(B)}=\widetilde{s}_{n+6}^{(B)}$. Since
$\widetilde{s}_6^{(A)}=\widetilde{s}_8^{(A)}=2$, further $\widetilde{s}_6^{(B)}=\widetilde{s}_8^{(B)}=-2$ (see Table \ref{table:altsum_of_values}) we conclude 
$\widetilde{s}_{n+6}^{(A)}=2$ and $\widetilde{s}_{n+6}^{(B)}=-2$. Thus
$$
\widetilde{s}_{n}=\widetilde{s}_n^{(A)}+\widetilde{s}_n^{(B)}+2=2
$$
holds for $n\ge6$. The smaller cases can be verified by hand, and the proof is complete.
\end{proof}

By the help of $\hat{s}_n$ and $\widetilde{s}_n$ we can easily determine the alternate sum with the arbitrary weights $v$ and $w$. %Let $n_{\equiv2}$ denote the remainder of $n$ divided by 2.
\begin{cor}
$$
\widetilde{s}_{(v,w),n}=\sum_{i=0}^{s_n-1}\left(v{\delta_{0, \,i\bmod2}}+w{\delta_{1,\, i\bmod2}}\right)\binomhh{n}{i}=\frac{\hat{s}_n+\widetilde{s}_n}{2}v+\frac{\hat{s}_n-\widetilde{s}_n}{2}w=
\frac{v+w}{2}\hat{s}_n+\frac{v-w}{2}\widetilde{s}_n,
$$
where $\delta_{j,i}$ is the Kronecker delta.
\end{cor}

\subsection{Location of elements type $A$ and $B$ in the hyperbolic Pascal triangle ${\cal I}_{4,5}$}

The aim of this subsection is to give the location of the elements type $A$ and $B$ within rows of the hyperbolic Pascal triangle belonging to the graph ${\cal I}_{4,5}$.
First we provide a description of the structure of ${\cal I}_{4,5}$ concentrating only on the type of the elements and not considering their value labelled (apart from Lemma \ref{lemma:rep}). In other words, we examine only the pattern displayed by the elements of type $A$ and $B$ (see Figure \ref{fig:AB}). 
\begin{figure}[h!]
 \centering
 \includegraphics[width=0.70\linewidth]{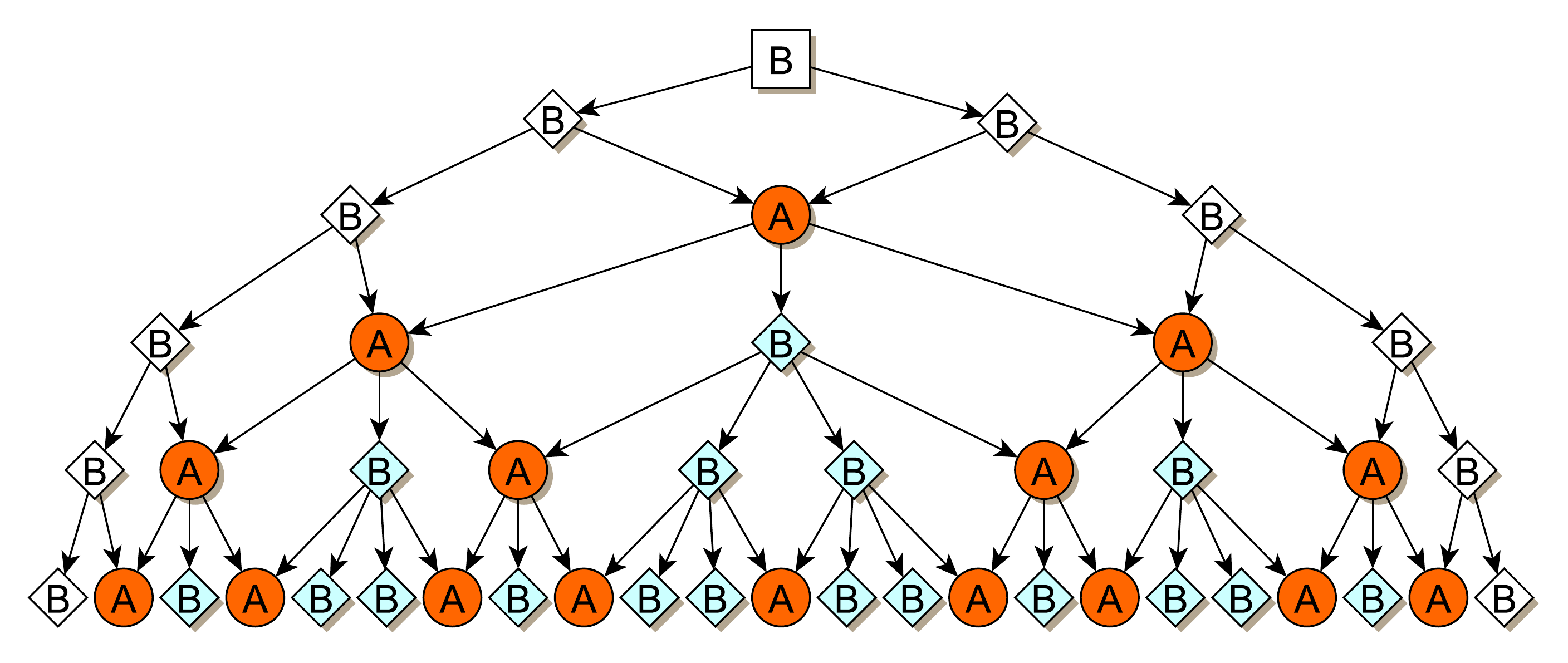}
 \caption{Pattern provided by the types $A$ and $B$ in ${\cal I}_{4,5}$}
 \label{fig:AB}
\end{figure}
Then we construct a sequence which involves the information about the rows in recursive manner. Note that here we always consider the winger elements of ${\cal I}_{4,5}$ as type $B$, which facilitates the general handling.

Recall, that ${\cal I}_{4,5}$ possesses a vertical symmetry. Now we collect some other important facts.

\begin{lemma}\label{lemma:rep}
Let $n=3k\ge3$. In row $n$ there is a central element of type $B$ (possessing four descendants), which is labelled with value $2^k$. This vertex is a root of a subgraph isomorphic to ${\cal I}_{4,5}$. 
\end{lemma}

%{\color{red} \bf Ha $n=0$, akkor is $1=2^0$. $3^i$ is szerepel mint subgraphs. $4^i$, $5^i$, ... is.}
\begin{proof}
We use the technique of induction. The first structural repetition of ${\cal I}_{4,5}$ can be observed in Figure \ref{fig:Pascal_layer5}, starting by the middle element 2 of row 3. Further, Figure \ref{fig:centralelement} makes it possible to deduce, that if the central element of type $B$ in row $3k$ has value $y$, then the central element in row $3(k+1)$ does $2y$, moreover its type is $B$ again. Finally, a suitably separated subgraph of ${\cal I}_{4,5}$ starting at the central element of row $3k$ has the following property: the leftmost (and the rightmost) elements has type $B$ (with four descendants, among them only two are included), consequently the excluded part of ${\cal I}_{4,5}$ does not affect the included one (no outer input).
\end{proof}

\begin{figure}[h!]
 \centering
 \includegraphics[width=0.55\linewidth]{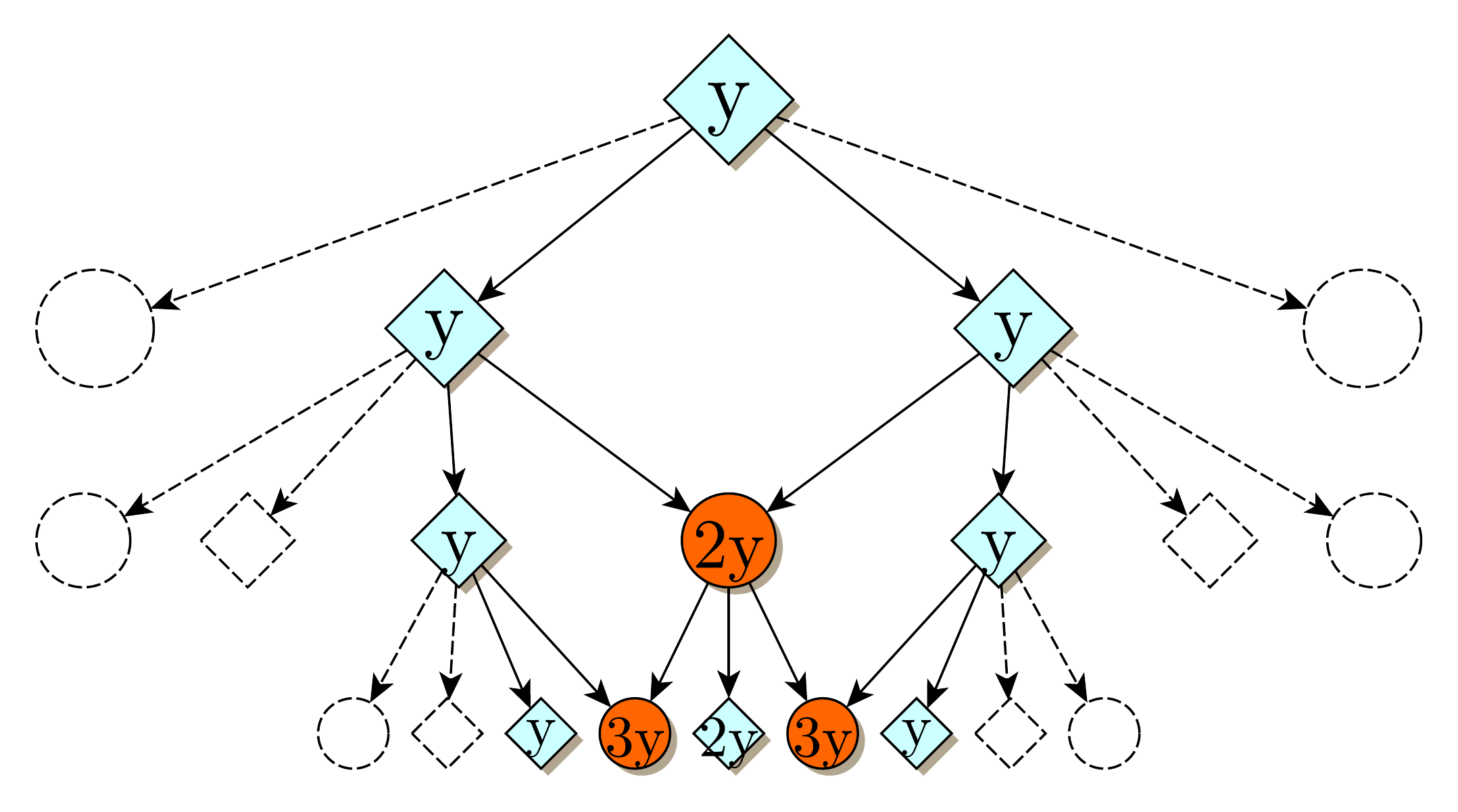}
 \caption{Central part of the Pascal triangle between rows $3k$ and $3(k+1)$}
 \label{fig:centralelement}
\end{figure}
\begin{cor}\label{c:rep}
The central part of row $n+3$ is a repetition of the whole row $n$ ($n\ge0$). 
\end{cor}
The reproductive phenomenon of Lemma \ref{lemma:rep} appears also at any element type $B$ which has a direct ascendant type $A$ in the previous row. Here we use the observation only for the central elements, later in Theorem \ref{lem:uv} we will refer it under general circumstances.
By Lemma \ref{lemma:rep} we know the middle part of the graph, since in row $3k$ the central element begins to reproduce the graph ${\cal I}_{4,5}$ itself (apart from the labelled values according to our deal). For further discussion, see the next lemma.

\begin{lemma}\label{lemma:R}
Let $n\ne1$. The pattern (of characters $A$ and $B$) given by the first $s_n$ elements of row $n+1$ coincide the pattern of row $n$.
\end{lemma}
%{\color{red} \bf Csak a cs\'ucspontok t\'{i}pusai. Minta -- pattern?}
\begin{proof}
For the row $n$ with $n\le3$ the statement is trivial (see Figure \ref{fig:AB}, or Figure \ref{fig:symmetry01}). Assume now that the property of Lemma \ref{lemma:R} is true for row $n$. In order to construct the whole row $n+1$ at the beginning of row $n+2$, it is sufficient to have the whole row $n$ at the beginning of row $n+1$. But it is obviously true by the assumption of the induction.
\end{proof}

\begin{cor}
By the vertical symmetry, analogous statement holds for the last $s_n$ elements of row $n+1$.
\end{cor}

\begin{cor}\label{cor:surprise}
Apart from the first few rows, it is clear that at the beginning of any row, let say row $n+1$, the part of the previous row ($n^{th}$) obtained by erasing the whole row $n-1$ from the end of row $n$ appears.
\end{cor}

Up to our present knowledge there is a ``central part'' of ${\cal I}_{4,5}$, and the left and right wings are isomorphic. 
%The central part of row $n$, by Lemma \ref{lemma:rep}, is a repetition of the whole row $n-3$ ($n\ge3$). 
The next lemma clarifies the structure of the left wing. 
\begin{lemma}\label{lemma:slice}
The left wing can be split into two subgraphs isomorphic to each other. One of the two is exactly determined in Corollary \ref{cor:surprise}.
\end{lemma}

\begin{proof}
Apply induction for the rows $n\ge5$ on Figure \ref{fig:symmetry01} to see the statement. 
\end{proof}

\begin{figure}[h!]
 \centering
 \includegraphics[width=0.99\linewidth]{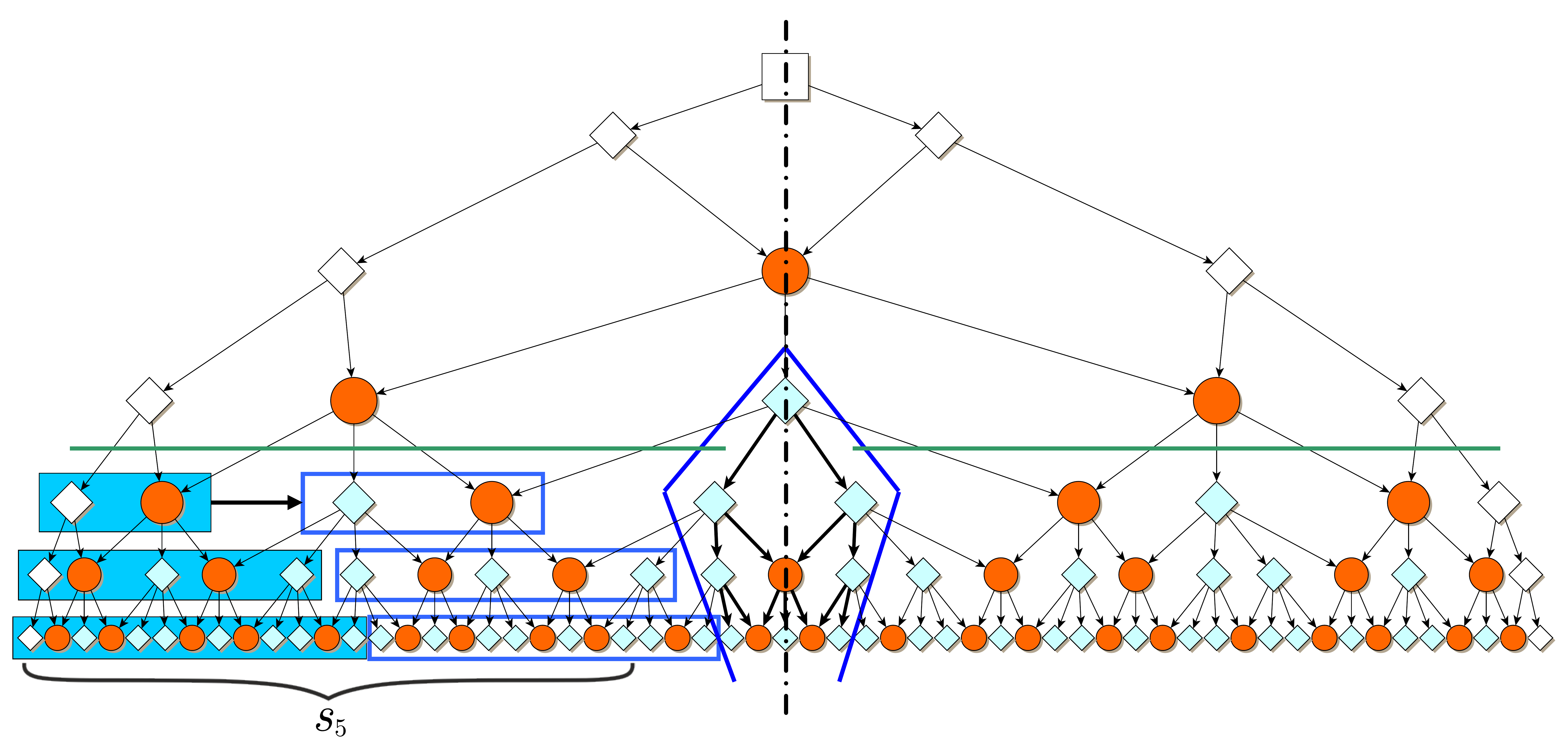}
 \caption{Symmetries of the hyperbolic Pascal triangle $\{4,6\}$}
 \label{fig:symmetry01}
\end{figure}
 
Now we can phrase the main result of this subsection. Let encode the type of the elements $A$ and $B$ by $0$ and $1$, respectively, and consider the code of row $n$ as a binary expansion of the integer $\vp_n$. For example, if $n=3$, then the pattern $BABAB$ of row $3$ is transformed into $10101$, hence $\vp_3=10101_{2}=21$.

\begin{theorem}
Put $S_n=2^{s_{n+1}-s_n}$, and $\Phi(n)=\vp_{n+1}-\vp_n$. The terms of the sequence $\Phi_n$ satisfy
$$
\Phi_n=\left(\frac{S_n}{S_{n-1}}+S_n+S_{n-1}\right)\Phi_{n-1}-S^2_{n-1}\Phi_{n-2},\qquad(n\ge3).
$$
\end{theorem}

\begin{proof}
The combination of Lemma \ref{lemma:slice}, Corollary \ref{cor:surprise} and the vertical symmetry imply (see Figure \ref{fig:symmetry01}) that
\begin{eqnarray*}
\vp_{n+1}&=&\left(\vp_n-2^{s_{n}-s_{n-1}}\vp_{n-1}\right)\left(2^{s_{n}-s_{n-1}}+1\right)+2^{2(s_{n}-s_{n-1})}\vp_{n-2} \\
&&+\frac{\vp_n-\vp_{n-1}}{2^{s_{n-1}}}\left(2^{2(s_n-s_{n-1})+s_{n-2}}+2^{3(s_n-s_{n-1})+s_{n-2}}\right).
\end{eqnarray*}
Applying the recursive formula of the sequence $\{s_n\}$ given by Theorem \ref{theorem:numvertex4q} and the definition of $S_n$ and $\Phi_n$, the proof is complete.
\end{proof}

\subsection{Occurrence of integers of certain types in hyperbolic Pascal triangle ${\cal I}_{4,5}$}

Since $\binomh{i}{1}=i$, it follows trivially that any positive integer appears in the triangle linked to ${\cal I}_{4,5}$. Therefore a more difficult question arises naturally, whether any pair of positive integers exists in the triangle as neighbour elements of a row. 

\begin{theorem}\label{lem:uv}
Given $u,v\in \mathbb{N}^+$, then there exist $i,j\in \mathbb{N}^+$ such that $u=\binomh{i}{j}$ and $v=\binomh{i}{j+1}$. 
\end{theorem}

\begin{proof}
There is no restriction in assuming $u\le v$ by the vertical symmetry of the triangle.
Clearly, if $u=1$, then $u=\binomh{v}{0}$ and $v=\binomh{v}{1}$. Moreover, if $u=v\ne1$, then choosing $u=\binomh{v+2}{4}$ and $v=\binomh{v+2}{5}$ is suitable. 

In the non-trivial cases, consider the Euclidean algorithm for $u$ and $v$ as follows. 
\begin{equation}\label{eq:Euclidean_alg}
\arraycolsep=1.8pt\def\arraystretch{1}
\begin{array}{rcrcll}
v   &=& r_0 u &+& t_1 & \quad(0<t_1<u),\\
u &= & r_1 t_1 &+& t_2 & \quad(0<t_2<t_1),\\
& \vdots & & &\\
t_{n-2}&=& r_{n-1} t_{n-1} &+& t_n & \quad(0<t_n<t_{n-1}),\\
t_{n-1}&=& r_{n} t_{n}. & &\qquad &\;
\end{array}
\end{equation}

If we assume $t_n=\gcd(u,v)=1$, then $t_n=\binomh{t_{n-1}}{0}$ and $t_{n-1}=\binomh{t_{n-1}}{1}$ 
%or $t_{n-1}=t_n\ne1$ then $t_n=\binomh{t_{n-1}+1}{5}$ and $t_{n-1}=\binomh{t_{n-1}+1}{6}$
.
We can follow step by step the reverse order of the Euclidean algorithm on the graph.
 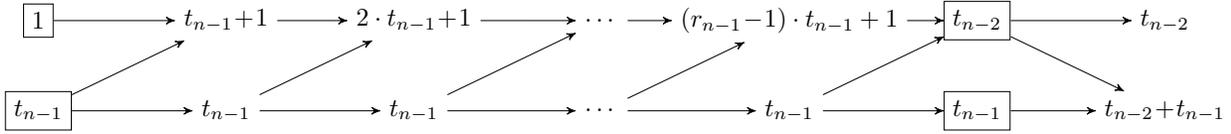
\begin{figure}[h!] \centering \footnotesize
\scalebox{0.96}{ \begin{tikzpicture}[->,xscale=2.6,yscale=2.5, auto,swap]
             % Draw the vertices.
             \node (a0) at (0,0) [shape=rectangle,draw] {$t_{n-1}$};
             \node (a1) at (1,0) {$t_{n-1}$};
             \node (a2) at (2,0) {$t_{n-1}$};
             \node (a3) at (3,0) {$\cdots $};
             \node (a4) at (4,0) {$t_{n-1}$};
             \node (a5) at (5,0)[shape=rectangle,draw] {$t_{n-1}$};
             \node (b0) at (0,0.5) [shape=rectangle,draw]{$1$};
             \node (b1) at (1,0.5) {$t_{n-1}\!+\!1$};
             \node (b2) at (2,0.5) {$2 \cdot t_{n-1}\!+\!1$};
             \node (b3) at (3,0.5) {$\cdots$};
             \node (b4) at (4,0.5) {$(r_{n-1}\!-\!1) \cdot t_{n-1}+1$};
             \node (b5) at (5,0.5)[shape=rectangle,draw] {$t_{n-2}$};

             \node (c1) at (6,0) {$t_{n-2}\!+\!t_{n-1}$};
             \node (c2) at (6,0.5) {$t_{n-2}$};
     
             % Connect vertices with edges and draw weights
             \path (a0) edge node {} (a1);
             \path (a1) edge node {} (a2);
             \path (a2) edge node {} (a3);
             \path (a3) edge node {} (a4);
             \path (a4) edge node {} (a5);
             \path (b0) edge node {} (b1);
             \path (b1) edge node {} (b2);
             \path (b2) edge node {} (b3);
             \path (b3) edge node {} (b4);
             \path (b4) edge node {} (b5);
             \path (a0) edge node {} (b1);   
             \path (a1) edge node {} (b2);                       
             \path (a2) edge node {} (b3);   
             \path (a3) edge node {} (b4);  
             \path (a4) edge node {} (b5);
             \path (a5) edge node {} (c1);   
             \path (b5) edge node {} (c2);  
             \path (b5) edge node {} (c1);                 
         \end{tikzpicture}}
  \caption{Illustration of the penult step of the Euclidean algorithm}
  \label{fig:last_Eucl}
 \end{figure}
 
Figure \ref{fig:last_Eucl} illustrates the penult Euclidean division $t_{n-2}=r_{n-1} t_{n-1} +1$. It shows that if $1$ and $t_{n-1}$ are neighbours, then $r_{n-1}$ rows hereinafter we find $t_{n-1}$ and $t_{n-2}$ as elements next to each other. Repeating this argument by going up in the steps of (\ref{eq:Euclidean_alg}), finally we will find the position of the neighbours $u$ and $v$ (by going down in the graph). If one switches from one step to the previous one in (\ref{eq:Euclidean_alg}), the side of the ``escorting'' elements of type $B$ is also changed (see the right end of Figure \ref{fig:last_Eucl}). Put $r=\sum_{k=0}^{n-1}r_k$. Then row $i=t_{n-1}+r$ is a suitable choice for finding the neighbours $u$ and $v$.

Assume now that $t_n>1$. Since $\binomh{t_n}{1}=t_n$ is an element of type $A$, it has three descendants, the middle one is $\binomh{t_n+1}{2}=t_n$. This element is an origin of a subgraph of ${\cal I}_{4,5}$ which structurally isomorphic to ${\cal I}_{4,5}$ itself (we quote the observation after Corollary \ref{c:rep}). But the values labelled to the appropriate descendants of $\binomh{t_n+1}{2}$ are $t_n$-multiple of ${\cal I}_{4,5}$. Here $i=r+2t_{n-1}/t_n+1$.
\end{proof}

\begin{theorem}\label{th:fn}
Let $\eta$, and  $f_0<f_1$ denote positive integers with $\gcd(f_0,f_1)=1$.
Further let $\{f_n\}$ denote a binary recurrence sequence given by 
$$f_{n}=\eta f_{n-1}+f_{n-2},\quad(n\geq2).$$ 
The values $f_0$ and $f_1$ appear next to each other in a suitable row of the Pascal triangle, such that the type of $f_1$ is $A$. Then all elements of the sequence are descendants of the vertex labelled by $f_1$, all have type $A$, moreover the distance of $f_j$ and $f_{j+1}$ ($j\ge1$) is $\eta$. 
\end{theorem}

\begin{proof}
Obviously, the sequence $\{f_n\}$ is strictly monotone increasing. In the virtue of Theorem \ref{lem:uv}, $f_0$ and $f_1$ exist next to each other. From the proof it is clear that the type of $f_1$ is $A$. By the defining rule of the sequence $\{f_n\}$, the Euclidean algorithm for determining the greatest common divisor of $f_n$ and $f_{n-1}$ ($n\ge2$) consists of the steps
$$
f_i=\eta f_{i-1}+f_{i-2},\qquad (i=n,n-1,\dots,2).
$$
To follow one step in the graph is analogous to the treatment in Theorem \ref{lem:uv}. In each step we go down exactly $\eta$ rows.
\end{proof}

Theorem \ref{th:fn} can be extended for any pair $f_0$ and $f_1$ if we combine it with Theorem \ref{lem:uv}.
 
The most interesting illustration of Theorem \ref{th:fn} is the Fibonacci sequence $\{F_n\}$. Now $\eta=1$, further we consider the sequence from $F_2=1$ and $F_3=2$.
Clearly, $\binomh{2}{0}=F_2$ and $\binomh{2}{1}=F_3$. The value $\eta=1$ implies that in each row the next Fibonacci number appears. From the proof of Theorem \ref{lem:uv} it follows that to have the sequence $\{F_n\}$ we need to use periodically the left-right descendants, starting from $F_2$. Since $\binomh{1}{0}=F_2$ and $\binomh{0}{0}=F_1$ also hold, we can see the whole Fibonacci sequence $\{F_n\}_{n\ge0}$ in the present hyperbolic Pascal triangle (see Figure \ref{fig:Pascal_Fibonacci}). An other well-known example, the Pell sequence $\{P_n\}_{n\ge0}$ defined by $P_1=1$, $P_2=2$ and $\eta=2$.  (For illustration see the same figure.)

\begin{figure}[h!]
 \centering
 \includegraphics[width=\linewidth]{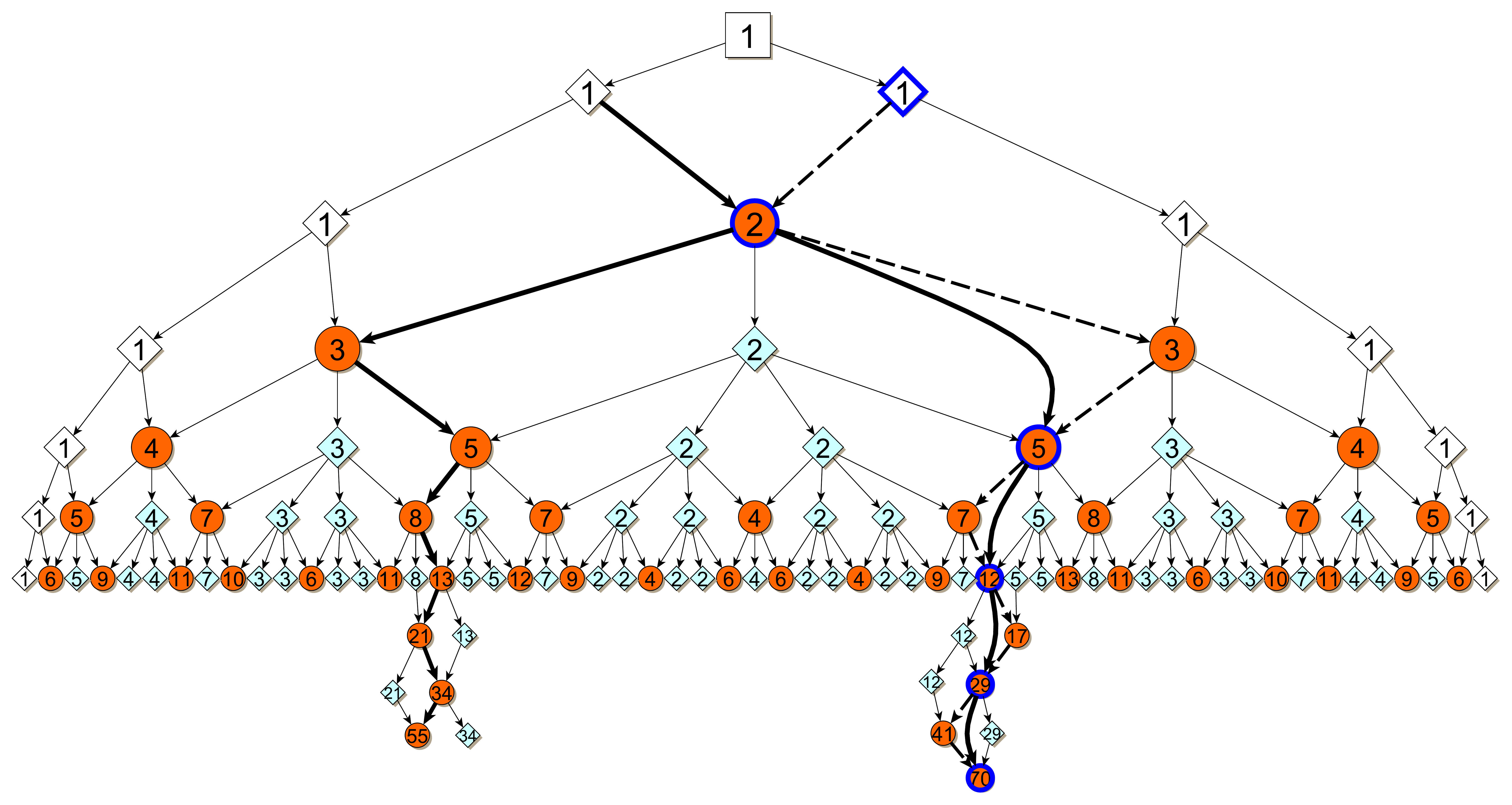}
 \caption{Fibonacci and Pell sequences in the hyperbolic Pascal triangle $\{4,5\}$}
 \label{fig:Pascal_Fibonacci}
\end{figure}

\section{A practical lemma}

\begin{lemma}\label{lemma:2seq}
Let $x_0$, $y_0$, further $a_i$, $b_i$ and $c_i$ ($i=1,2$) be complex numbers such that $a_2b_1\ne0$. Assume that
the for $n\ge n_0$ terms of the sequences $\{x_n\}$ and $\{y_n\}$ satisfy
\begin{eqnarray*}%\label{mix}
x_{n+1}&=&a_1x_n+b_1y_n+c_1,\\
y_{n+1}&=&a_2x_n+b_2y_n+c_2.
\end{eqnarray*}
Then for both sequences
\begin{equation}\label{nomix}
z_{n+3}=(a_1+b_2+1)z_{n+2}+(-a_1b_2+a_2b_1-a_1-b_2)z_{n+1}+(a_1b_2-a_2b_1)z_n
\end{equation}
holds ($n\ge n_0$).
\end{lemma}
\begin{proof}
Divide the first equation by $b_1$, eliminate $y_n$, and then plug the result into the second equation, multiply it by $b_1$ to have 
\begin{equation}\label{lastone}
x_{n+2}=(a_1+b_2)x_{n+1}+(-a_1b_2+a_2b_1)x_n+(b_1c_2-b_2c_1+c_1).
\end{equation}
In order to eliminate the constant term, we subtract it from the equivalent equation for $x_{n+3}$. Then we get immediately (\ref{nomix}).

One can show the same argument for $y_{n+3}$ by a similar manner.
\end{proof}

\begin{rem}
If both $c_1$ and $c_2$ are zero, then (\ref{nomix}) simplifies to (\ref{lastone}) since its constant term now is zero.
\end{rem}

\end{document}